\documentclass[10pt, article]{amsart}
\usepackage{tikz}
\usetikzlibrary{calc}
\usepackage{ae} 
\usepackage[T1]{fontenc}
\usepackage[cp1250]{inputenc}
\usepackage{amsmath}
\usepackage{amssymb, amsfonts,amscd,verbatim}

\usepackage[normalem]{ulem}
\usepackage{hyperref}
\usepackage{indentfirst}
\usepackage{latexsym}
\input xy
\xyoption{all}

\usepackage{xcolor}

\usepackage{amsmath}    

\theoremstyle{plain}
\newtheorem{Pocz}{Poczatek}[section]

\newtheorem{Conjecture}[Pocz]{Conjecture}
\newtheorem{Theorem}[Pocz]{Theorem}
\newtheorem{Corollary}[Pocz]{Corollary}

\newtheorem{Lemma}[Pocz]{Lemma}

\newtheorem{Question}[Pocz]{Question}

\theoremstyle{definition}
\newtheorem{Definition}[Pocz]{Definition}

\theoremstyle{remark}

\newtheorem{Exercise}[Pocz]{Exercise}

\numberwithin{equation}{section}
\title[Splitting of homotopy idempotents revisited]
{Splitting of homotopy idempotents revisited}

\author{Jerzy Dydak}
\address{University of Tennessee, Knoxville, TN 37996, USA}
\email{jdydak@utk.edu}

\date{ \today
}


\begin{document}
\maketitle
\begin{center}

    \large Piotr Minc, in memoriam

\today
\end{center}

\tableofcontents

\begin{abstract}
We are presenting proofs of fundamental results related to homotopy idempotents, proofs that are sufficiently simple so that even the author can understand them. 
The first one is that homotopy idempotents in the category of pointed connected CW complexes split and the second one is that unpointed homotopy idempotents in the category of finite-dimensional CW complexes split. Some of our proofs rectify gaps in the existing literature.
\end{abstract}

\section{Introduction}

\begin{Definition}
A \textbf{homotopy idempotent} on a CW complex $K$ is a map $f:K\to K$ such that $f\circ f$ is homotopic to $f$. It is said to \textbf{split} if there are maps $d:K\to L$ and $u:K\to L$ such that $d\circ u\sim id_L$ and $f\sim u\circ d$.
\end{Definition}
The above definitions can be made in both the category of pointed connected CW complexes and in the category of unpointed CW complexes.

It
 is well known (see E.M. Brown \cite{Br}, D.A.Edwards and R.Geoghegan \cite{EG}, P.Freyd \cite{Fr}) that pointed homotopy idempotents on pointed connected CW complexes split. However, those proofs involve the so-called Brown Representability Theorem and Matthew Brin \cite{Brin} asked this author if there is a simpler proof of that fact. That prompted the author to look at the existing literature on the subject. To his surprise he discovered that many published papers contain overly optimistic statements (lifts of homotopy idempotens are homotopy idempotents in \cite{HH}), and the essential fact that the natural homomorphism from the Thompson group must be a monomorphism if the idempotent does not split has not been published yet despite being used in most papers on the subject.
Also, the crucial paper \cite{HH} has so many misprints (lack of the tilde over many spaces, elements of a spectral sequence that need to be trivial but are labeled as non-trivial, mispelled representation of the Thompson group) that actually it is hard to follow it.

J.Dydak and P.Minc \cite{JD}, and Freyd and Heller \cite{FH}  independently found
 an unpointed homotopy idempotent on an infinite-dimensional complex which does
 not split. That means there is a fundamental difference between the pointed and unpointed categories, so the paper is divided into parts corresponding to those categories.

The new aspect of the paper is posing two questions: 
\ref{MyConjecture} and \ref{UniversalGroupConjecture}.

The author wishes to thank Matt Brin \cite{Brin} for posing a question of existence of a proof that pointed homotopy idempotents split without appealing to the Brown Representability Theorem.

This paper is dedicated to the memory of Piotr Minc (January 9, 1949 - September 24, 2023).

\section{Geometric aspect of idempotents}\label{Infinite telescope}
In this section, given a homotopy idempotent $f:K\to K$, we construct a geometric object that is the primary guess (see Exercise \ref{Exercise}) for the target $K\to L$ in splitting of $f$.

\begin{Definition}
Given a sequence of pointed cellular maps $f_n:K_n\to K_{n+1}$, $-\infty < n < \infty$, on connected CW complexes $K_n$, its \textbf{infinite telescope} $Tel(\{f_n\})$ is the infinite union $M_n$, $-\infty< n < \infty$, where $M_n$ is the \textbf{unreduced} mapping cylinder of $f_n:K_n\to K_{n+1}$.
\end{Definition}

Notice that $Tel(\{f_n\})$ has a natural subset $R$ homeomorphic to the reals that we will use as a base for free homotopy groups of $Tel(\{f_n\})$ (see Section \ref{FreeHomotopyGroups}).

\begin{Lemma}\label{UniversalCoverOfTel}
If all $K_n$ are connected, then the natural map
$p:Tel(\{\tilde f_n\})\to Tel(\{f_n\})$ lifts to a homotopy equivalence
$g:Tel(\{\tilde f_n\})\to \widetilde{Tel(\{f_n\})}$.
\end{Lemma}
\begin{proof}
It suffices to note that $g$ induces isomorphisms of all higher homotopy groups. If $h:S^k\to Tel(\{\tilde f_n\})$, $k\ge 2$, gets mapped to $0$,
then it can be pushed to some high $\tilde K_n$, so that
$p_n\circ h$ is null-homotopic in $K_n$. Therefore $h$ is null-homotopic.

Conversely, if $h:S^k\to Tel(\{f_n\})$, then it can be pushed in to some high
$K_n$ and then lifted.
\end{proof}

From now on, 
given a pointed cellular map $f:K\to K$ on a connected CW complex $K$, its \textbf{infinite telescope} $Tel(f)$ is the infinite telescope of $f_n:K_n\to K_{n+1}$, where $K_n=K$, $f_n=f$ for each $n$.

\begin{Definition}\label{MapsRelTel}
Given a pointed cellular map $f:K\to K$ on a connected CW complex $K$ and a (not necessarily pointed) homotopy $H:K\times I\to K$ joining $f$ and $f^2$, we define two maps
$d:K\to Tel(f)$ and $u:Tel(f)\to K$ as follows:\\
1. $d=f$ from $K$ to $K_0$,\\
2. on each $K_n$, $u=f$, and $u|M_n$ is simply induced by $H$.
\end{Definition}

\begin{Exercise}\label{Exercise}
Show that, if $f:K\to K$ is a homotopy idempotent that splits as $K\to L$ and $L\to K$, then $L$ is homotopy equivalent to $Tel(f)$.
\end{Exercise}

The next two lemmata illustrate why switching to $Tel(f)$ makes sense in order to determine if $f$ splits. They are also the basis for Conjecture \ref{MyConjecture}.

\begin{Lemma}\label{FirstGeometricObstruction}
Suppose $f:K\to K$ is a homotopy idempotent and $H:K\times I\to K$ is a homotopy joining $f$ and $f^2$. $g=d\circ u:Tel(f)\to Tel(f)$ is a homotopy idempotent and, if it splits, then so does $f$.
\end{Lemma}
\begin{proof}
Notice $d\circ f\sim d$, so $d\circ f^2\sim d$. Therefore,
$g\circ g= d\circ u\circ d\circ u= d\circ f^2\circ  u\sim d\circ u=g$.

Suppose $ d\circ u\sim u_1\circ d_1$, $d_1:Tel(f)\to L$ and $u_1:L\to Tel(f)$,
so that $u_1\circ d_1\sim g$ and $d_1\circ u_1\sim id_L$. Let $u_2:L\to K$ be equal to $u\circ u_1$
and let $d_2:K\to L$ be equal to $d_1\circ d$.

Now, $d_2\circ u_2= d_1\circ d\circ u\circ u_1\sim d_1\circ g\circ u_1 
\sim d_1\circ u_1\circ d_1 \circ u_1\sim id_L\circ id_L=id_L $
and $u_2\circ d_2= u\circ u_1\circ d_1\circ d\sim u\circ g\circ d=
u\circ d\circ u \circ d\sim f^2\circ f^2\sim f$.

\end{proof}

\begin{Lemma}\label{FiniteComplexesLemma}
Suppose $f:K\to K$ is a homotopy idempotent and $H:K\times I\to K$ is a homotopy joining $f$ and $f^2$. For every map $h:L\to Tel(f)$, $d\circ 
u \circ h$ is homotopic to $h$ if $L$ is a finite CW complex.
\end{Lemma}
\begin{proof}
Let $d_n:K\to Tel(f)$ be the same as $f:K\to K_n$. Thus, $d=d_0$.

We may assume that $h(L)\subset K_n$ for some $n > 0$ as $h(L)$ is contained in a finite union of mapping cylinders $M_k$ and their union has a deformation retraction to one of $K_n$'s. Now, $g$ can be pushed in $Tel(f)$ to $K_{n+1}$  and in that range 
$d_{n+1}\circ 
u \circ h=f^2\circ h\sim f\circ h$, so $d\circ 
u \circ h\sim h$.
\end{proof}

\begin{Corollary}\label{PointedIdempotents}
Given a a cellular map $f:K\to K$ on a pointed connected CW complex $K$ and a pointed homotopy $H:K\times I\to K$ joining $f$ and $f^2$, $d\circ 
u :Tel(f)\to K$ is homotopic to the identity on $Tel(f)$. In particular, $f$ splits as
$f\sim u\circ 
d=f^2$.
\end{Corollary}
\begin{proof}
By Lemma \ref{FiniteComplexesLemma} in the pointed case (one can use the infinite telescope composed of reduced mapping cylinders of $f$), $ d\circ 
u$ induces the identity on all homotopy groups of $Tel(f)$. Hence, $d\circ 
u$ is a homotopy equivalence by the Whitehead Theorem. However, as in the proof of Lemma \ref{FirstGeometricObstruction},
$$ (d\circ u)\circ (d\circ u)=d\circ f^2\circ u\sim d\circ 
u$$
resulting in
$$ d\circ 
u\sim id_{Tel(f)}.$$
\end{proof}

\section{Algebraic aspects of idempotents}
In this section we summarize the algebraic results related to unpointed homotopy idempotents that are of interest to us. We cite purely algebraic results \ref{FreeAbelianSubgroupsOfF}, \ref{L4Heller-Freyd}, and we show how to use them in connection to conjugate idempotents of groups.

First of all, recall that by the Thompson group $F$ we mean the group
with generators $a_i, i\ge 0$, and relations $a_i^{-1}\cdot a_j\cdot a_i=a_{j+1}$ for all $i < j$. The \textbf{shift homomorphism} $s:F\to F$ satisfies $s(a_i)=a_{i+1}$ for all $i\ge 0$.

The easiest way to prove that, in the case of unsplittable homotopy idempotent $f:K\to K$, there is a monomorphism $F\to \pi_1(K)$ is to use the fact that the image of any $h:F\to G$ is Abelian if $h$ is not a monomorphism. That fact can be found in \cite{FH} but it is too involved for our taste.

See \cite{JD} for a representation of $F$ due to Piotr Minc that is handy in proving the following (see also \cite{HH2}):
\begin{Lemma}\label{FreeAbelianSubgroupsOfF}
Elements $a_{3i}^{-1}\cdot a_{3i+1}$, $i\ge 0$, of $F$ commute and are linearly independent.
\end{Lemma}

Notice that, given an endomorphism $f:G\to G$ of a group $G$
such that $x_0^{-1}\cdot g(x)\cdot x_0$ for some $x_0\in G$ and all $x\in G$,, there is a natural homomorphism $e:F\to G$ defined by
$$e(a_0)=x_0, e(a_k)=(f^k)(x_0), k\ge 1.$$

The following is a purely algebraic result whose clear proof one can find in \cite{FH}  (L4 on p.99):
\begin{Lemma}[L4 of Freyd-Heller \cite{FH}]\label{L4Heller-Freyd}
Any non-trivial element of $F$ is conjugate to an element
of the form $a_i^n\cdot s^{i+1}(b)$ for some $n\ne 0$ and $i=0,1$.
\end{Lemma}

\begin{Lemma}[L2 of Freyd-Heller \cite{FH}]\label{L2Heller-Freyd}
Suppose $f:G\to G$ is an endomorphism of a group $G$ such that
$f^2(x)=x_0^{-1}\cdot f(x)\cdot x_0$ for some $x_0$ and all $x\in G$.
If $x_0=f(y)$ for some $y\in G$, then 
$f$ is conjugate to $g:G\to G$ such that $g^2=g$.
\end{Lemma}
\begin{proof}
Let $g(x)=y\cdot f(x)\cdot y^{-1}$ for all $x\in G$.
Now, $g^2(x)=g(y)\cdot g(f(x))\cdot g(y^{-1})=
g(y)\cdot y\cdot f^2(x)\cdot y^{-1}\cdot g(y^{-1})=
y\cdot f(y)\cdot x_0^{-1}\cdot f(x)\cdot x_0\cdot f(y^{-1})\cdot y^{-1}=
y\cdot f(x)\cdot y^{-1}=g(x)$.
\end{proof}

\begin{Lemma}\label{GeneralConjugation}
Suppose $f:G\to G$ is an endomorphism of a group $G$ such that
$f^2(x)=x_0^{-1}\cdot f(x)\cdot x_0$ for some $x_0$ and all $x\in G$.
Let $x_i=f^i(x_0)$ for $i > 0$.
If $m > i$ then
$$x_i^{-1}\cdot f^m(x)\cdot x_i=f^{m+i+1}(x)$$
for all $x\in G$.
\end{Lemma}
\begin{proof}
Let us proceed by induction on $i$. 

If $i=0$, then
$$x_0^{-1}\cdot f^m(x)\cdot x_0=x_0^{-1}\cdot f(f^{m-1}(x))\cdot x_0=f^{m+1}(x).$$
If 
$$x_i^{-1}\cdot f^m(x)\cdot x_i=f^{m+i+1}(x)$$
then by applying $f$ to both sides we get the inductive step
$$x_{i+1}^{-1}\cdot f^{m+1}(x)\cdot x_{i+1}=f^{m+1+i+1}(x).$$
\end{proof}

\begin{Corollary}\label{SuperConjugation}
Suppose $f:G\to G$ is an endomorphism of a group $G$ such that
$f^2(x)=x_0^{-1}\cdot f(x)\cdot x_0$ for some $x_0$ and all $x\in G$.
Let $x_i=f^i(x_0)$ for $i > 0$.
If $m > i$ then 
$$x_i^{-k}\cdot f^m(x)\cdot x_i^k=f^{m+(i+1)\cdot k}(x)$$
for all $x\in G$ and all $k\ge 1$.
\end{Corollary}
\begin{proof}
$k=1$ is done in \ref{GeneralConjugation}. Use induction and \ref{GeneralConjugation} by conjugating with $x_i$.

\end{proof}

\begin{Corollary}\label{BasicIdempotentResult}
Suppose $f:G\to G$ is an endomorphism of a group $G$ such that
$f^2(x)=x_0^{-1}\cdot f(x)\cdot x_0$ for some $x_0$ and all $x\in G$.
Let $x_i=f^i(x_0)$ for $i > 0$.
If there is $k\ge 1$ and $i\ge 0$ such that
$$x_i^k\in im(f^{i+1})$$
then some power of $f$ is conjugate to $g:G\to G$ satisfying $g^2=g$.
\end{Corollary}
\begin{proof}
By \ref{SuperConjugation} (conjugate both sides with $x_i$ several times)
$$x_i^k\in im(f^{i+m})$$
for all $m\ge 1$. By applying $f$ to both sides we may assume $i\ge 1$.
Now, put $n=k\cdot (i+1)$ to see that
$$x_i^{-k}\cdot f^n(x)\cdot x_i^k=f^{2n}(x)$$
By \ref{L2Heller-Freyd}, $f^{n}$ is conjugate to $g$ such that $g^2=g$.
\end{proof}

\begin{Lemma}[Freyd-Heller \cite{FH}]\label{T1Heller-Freyd}
Suppose $f:G\to G$ is an endomorphism of a group $G$ such that
$f^2(x)=x_0^{-1}\cdot f(x)\cdot x_0$ for some $x_0$ and all $x\in G$.
If the canonical homomorphism $e:F\to G$ from the Thompson group $F$ is not a monomorphism, then there is $n\ge 1$ such that
$f^n$ is conjugate to $g:G\to G$ such that $g^2=g$.
\end{Lemma}
\begin{proof}
By \ref{L4Heller-Freyd}, the kernel of $e$  contains an element of the form
$a_i^n\cdot s^{i+1}(b)$ for some $n\ne 0$ and $i=0,1$.
Use \ref{BasicIdempotentResult}.
\end{proof}

\section{Unpointed homotopy idempotents}
In this section we give a geometric condition for an unpointed homotopy idempotent to split and an algebraic obstruction for it to split. Moreover, we can see how they interact.

\begin{Lemma}\label{SufficientConditionForSplitting}
Suppose $f:K\to K$ is a pointed map and $H:K\times I\to K$ is a homotopy from $f$ to $f^2$
in which the base point traverses a loop $\omega$.
If there is a loop $\alpha$ such that $[\omega]=[f\circ \alpha]$ in $\pi_1(K)$, then $f$ splits in the unpointed homotopy category of CW complexes.
\end{Lemma}
\begin{proof}
Let $g:K\to K$ be a map such that for some homotopy $G:K\times I\to K$ from $f$ to $g$, the base point traverses $\alpha^{-1}$, the reverse of $\alpha$. If we splice four homotopies:
$G^{-1}$, $H$, $(f\times id)\circ G$, and $G\circ (g\times id)$, we get a homotopy from $g$ to $g^2$ in which the base point traverses the concatenation of
the following loops: $\alpha$, $\omega$, $(f\circ \alpha)^{-1}$, and $\alpha^{-1}$. That concatentation is homotopically trivial, so $g^2$ is homotopic to $g$ in the pointed homotopy category of connected CW complexes.
By \ref{PointedIdempotents}, $g$ splits in the pointed homotopy category, so $f$ splits in the unpointed homotopy category.
\end{proof}

\begin{Lemma}\label{AlgebraicObstruction}
Suppose $f:K\to K$ is an unpointed homotopy idempotent on a connected CW complex. If $f$ does not split, then the induced homomorphism $F\to \pi_1(K)$ of the Thompson group $F$ to $\pi_1(K)$ is a monomorphism.
\end{Lemma}
\begin{proof}
Suppose the induced homomorphism $F\to \pi_1(K)$ of the Thompson group $F$ to $\pi_1(K)$ is not a monomorphism. By \ref{T1Heller-Freyd}, $f$ is homotopic to $h$ such that
$\pi_1(h^2)=\pi_1(h)$.
Let $g=d\circ u:Tel(h)\to Tel(h)$ as in \ref{MapsRelTel}. It is a homotopy idempotent that does not split by \ref{FirstGeometricObstruction}.
Notice $\pi_1(g)$ is surjective if $\pi_1(h^2)=\pi_1(h)$, so by \ref{SufficientConditionForSplitting},
$g$ splits, a contradiction.

\end{proof}

\section{Homotopy idempotents on complexes of finite dimension}

Hastings and Heller \cite{HH}  proved that homotopy idempotents on CW complexes $K$ of finite dimension split (the 2 dimensional case of it was proven earlier by Dydak-Hastings \cite{DH}). Both papers are really considering only the case of the natural homomorphism from the Thompson
group $F$ to $\pi_1(K)$ being a monomorphism. In this section we present a proof in the general case that also simplifies the proof in \cite{HH}.

\begin{Lemma}\label{DimensionLemma}
Let $K$ be a connected CW complex and $\pi_1(K)=\mathbb{Z}^n$ for some $n\ge 1$. 

If $\pi_1(K)$ acts trivially on an element $x\ne 0$ of the highest non-trivial homology group $H_r(\tilde K)\ne 0$ of the universal cover $\tilde K$ of $K$, then $\dim(K)\ge n+r$.
\end{Lemma}
\begin{proof}
Consider the spectral sequence from Section 10.8 of \cite{HW} (pp.464--468)
for $p:\tilde K\to K$, especially how it is used in the proof of Theorem 10.8.7 which is related to our case of $n=1$. Its terms are
$$E^2_{p,q}=H_p(\pi_1(K);H_q(\tilde K))$$
with integral coefficients and it converges to $H_{p+q}(K)$.
Hence, $E^2_{p,q}=0$ if $p > n$ or $q > r$.

If we look at $E^2_{n,r}=H_n(\pi_1(K);H_r(\tilde K))$, then
$E^2_{n,r}=E^\infty_{n,r}=H_n(\pi_1(K);H_r(\tilde K))\ne 0$, so $H_{n+r}(K)\ne 0$ and $\dim(K)\ge n+r$.
\end{proof}

\begin{Question}
The proof of Lemma \ref{DimensionLemma} uses a spectral sequence that is rarely applied in contrast to the Serre spectral sequence. Is there a more elementary tool that can be used there?
\end{Question}

\begin{Theorem}\label{MainResult}
If $f:K\to K$ is an unpointed homotopy idempotent on a finite-dimensional CW complex, then it splits.
\end{Theorem}
\begin{proof}
First, assume $K$ is connected.
Suppose $f$ does not split.

Now, we may assume $f$ preserves a base point.
Finally, we may assume that the natural homomorphism $e:F\to \pi_1(K)$ is a monomorphism 
by \ref{AlgebraicObstruction}. Notice that the composition
$F\to \pi_1(K)\to \pi_1(Tel(f))$ is also a monomorphism by using \ref{T1Heller-Freyd}.

Let $m=\dim(K) < \infty$.

Notice all $a_i$, $i < 3m+6$, act the same way on any element
$\omega$ in the image of $H(\tilde K_n)\to H(Tel(\tilde f))$ for $n > 3m+6$. 
Indeed,
the meaning of
\ref{GeneralConjugation} is that $\tilde f(a_i\cdot (\tilde f)(\omega))=
a_{i+1}\cdot (\tilde f)(\omega)$, so both
$a_i\cdot (\tilde f)(\omega)$ and $a_{i+1}\cdot (\tilde f)(\omega)$ are homologous.

The images of $H(\tilde K_n)\to H(Tel(\tilde f))$ cannot be trivial as in that case $Tel(\tilde f))$ is contractible and the reasoning below leads to a contradiction.

By \ref{FreeAbelianSubgroupsOfF}, there is a subgroup
$H\sim \mathbb{Z}^{m+2}$ of $\pi_1(Tel(f))$ that acts trivially 
on an element of the homology of $Tel(\tilde f)$. Now, there is a covering (up to homotopy, see \ref{UniversalCoverOfTel})
$Tel(\tilde f)\to L$, where $L\to Tel(f)$ is a covering corresponding to $H$. Therefore, $\dim(L)\ge m+2$, a contradiction.

If $K$ is not connected but $f(C)\subset C$ for each component $C$ of $K$,
then we can split $f|C$ as $d_C:C\to L_C$, $u_C:L_C\to C$ for each $C$
and define $L$ as the disjoint union of all $L_C$, $C$ a component if $K$.
It is easy to see that splicing all $d_C$ gives $d:K\to L$ and splicing all
$u_C:L_C\to C$ gives $u:L\to K$ that show a splitting of $f$.

If $f(C)$ is not subset of $C$ for some components $C$ of $K$, then we may
split $K$ as the disjoint union $K_0\oplus K_1$, where $K_1$ is the union of all components $C$ of $K$ such that $f(C)\subset C$ and $K_1$ is the union of all the remaining components. Notice that $f(K_1)\subset K_0$, so $K_0$ is not empty.

Let $d_0:K_0\to L$ and $u_0:L\to K_0$ be a splitting of $f|K_0:K_0\to K_0$. Define $u:L\to K$ as $u_0$.
Define $d:K\to L$ as $d|K_0=d_0$ and $d|C=d_C\circ f|C$,
where $d_C=d_D$, $D$ being the component of $K$ containing $f(C)$.
Obviously, $d\circ u\sim id_L$ and $u\circ d|K_0\sim f|K_0$.
Given a component $C$ of $K_1$ and the component $D$ of $K$ containing $f(C)$,
 $u\circ d_C=u_D\circ d_D\circ f|C\sim (f|D)\circ (f|C)=f^2|C\sim f|C$.
That means $f$ splits.
\end{proof}

Notice, in view of \ref{FiniteComplexesLemma}, that the following conjecture, if valid, would be stronger than Theorem \ref{MainResult}:
\begin{Conjecture}\label{MyConjecture}
If $K$ is a CW complex of finite dimension and $f:K\to K$ has the property that $f\circ g\sim g$ for every map $g:L\to K$ such that $L$ is a finite CW complex, then $f$ is a homotopy equivalence.
\end{Conjecture}

\begin{Lemma}
If $f:G\to G$ is an endomorphism of a free groups such that for any homomorphism
$g:H\to G$ from a finitely generated free group $H$ the composition $f\circ g$ is conjugate to $g$, then $f$ is conjugate to $id_G:G\to G$.
\end{Lemma}
\begin{proof}
Obviously, it is true if $G$ is finitely generated, so assume it is not finitely generated.
Enumerate free generators of $G$ as $x_s$, $s\in S$. For each
$s\in S$ there is $a(s)\in G$ such that $g(x_s)=a(s)^{-1}\cdot x_s\cdot a(s)$. We choose the one for which the smallest number $n(s)$ of generators is needed to express it (that is not the same as the length of $a(s)$). 

Given $s,t\in S$, $s\ne t$, there is $b\in G$ such that
$$g(x_s)=b^{-1}\cdot x_s\cdot b$$
and 
$$g(x_t)=b^{-1}\cdot x_t\cdot b.$$
That implies $b\cdot a(s)^{-1}$ is a power of $x_s$ as it commutes with $x_s$. Consequently,
$$n(s)\leq n(b)\leq n(s)+1.$$
Similarly,
$$n(t)\leq n(b)\leq n(t)+1.$$
That means there are at most two values in the set $\{n(s)\}_{s\in S}$.

Now, pick $z\in S$ and define $h:G\to G$ as $h(x)=a(z)\cdot g(x)\cdot a(z)^{-1}$. $h$ has the same crucial property as $g$ but the new function $n$ for $h$ has $n(z)=0$, so $n(t)\leq 1$ for all $t\in S$.

Given $s\in S$, $s\ne z$, there is $b\in G$ such that
$$h(x_s)=b^{-1}\cdot x_s\cdot b$$
and 
$$x_z=h(x_z)=b^{-1}\cdot x_z\cdot b.$$
That implies $b$ is a power of $x_z$ as it commutes with $x_z$. 

Suppose there are $s,t\in S\setminus\{z\}$ such that
$h(x_s)=x_z^{-k}\cdot x_s\cdot x_z^k$ and $h(x_t)=x_z^{-m}\cdot x_s\cdot x_z^m$, where $m\ne k$. This time choose $b$ as above but for three values $s,t,z$ of $S$ to arrive at a contradiction.

Thus, $h(x)=x_z^{-k}\cdot x\cdot x_z^k$ for some $k$ and all $x\in G$.
\end{proof}

\begin{Corollary}
Suppose $K$ is a CW complex and $f:K\to K$ has the property that $f\circ g\sim g$ for every map $g:L\to K$ such that $L$ is a finite CW complex. If $\dim(K)\leq 1$, then $f$ is a homotopy equivalence.
\end{Corollary}

\begin{Question}\label{UniversalGroupConjecture}
Is there a universal group that measures an obstruction of $f:K\to K$ to
be a homotopy equivalence if $f:K\to K$ has the property that $f\circ g\sim g$ for every map $g:L\to K$ such that $L$ is a finite CW complex?
\end{Question}

\section{Free homotopy groups}\label{FreeHomotopyGroups}

In this section we introduce free homotopy groups in order to avoid picking a base point in the telescope $Tel(f)$. $Tel(f)$ has a natural contractible subset $A$ that should serve as its base subset.

\begin{Definition}
Given a non-empty path-connected subset $A$ of a topological space $X$ such that any map
$f:S^1\to A$ is null-homotopic in $X$, we define the \textbf{free fundamental group} $\pi_1(X,A)$ as follows:\\
1. Its elements are homotopy classes of maps $a:(I,\partial I)\to (X,A)$, where $a$ is homotopic to $b$ if there are paths $u, v$ in $A$ joining
$a(1)$ to $b(0)$ and $b(1)$ to $a(1)$ such that the path
$a\ast u\ast b\ast v$ is homotopic to $a$ relatively to endpoints.\\
2. The product $[a]\ast [b]$ of two homotopy classes is $[c]$,
where $c=a\ast v\ast b$, $v$ being a path in $A$ joining $a(1)$ and $b(0)$.
\end{Definition}

\begin{Exercise}
The definition of $\pi_1(X,A)$ is sound.
\end{Exercise}

\begin{Exercise}
$\pi_1(X,A)$ is a group.
\end{Exercise}

\begin{Exercise}
If $B\subset A$, then the natural function $\pi_1(X,B)\to \pi_1(X,A)$ is an isomorphism. In particular, $\pi_1(X,x_0)\to \pi_1(X,A)$ is an isomorphism
for any $x_0\in A$.
\end{Exercise}

\begin{Corollary}
If $A\cap B\ne\emptyset$, then $\pi_1(X,A)$ is isomorphic
to $\pi_1(X,B)$.
\end{Corollary}

\begin{Exercise}
If $X$ is path-connected, locally path-connected, and semi-locally simply connected, then the universal cover $\tilde X$ of $X$ can be represented as the space of homotopy classes
of maps $a:(I,0)\to (X,A)$ relative $1$.

The action of $\pi_1(X,A)$ on $\tilde X$ can be described as
$[a]\ast [x]=[a\ast x]$.
\end{Exercise}

\begin{Definition}
Let $n > 1$.
Given a non-empty path-connected subset $A$ of a topological space $X$ such that any map
$f:S^k\to A$, $k\leq n$, is null-homotopic in $X$, we define the \textbf{free homotopy group} $\pi_n(X,A)$ as follows:\\ 
1. Its elements are homotopy classes of maps $a:(I^n,\partial I^n)\to (X,A)$, where $a$ is homotopic to $b$ if there are homotopies $u, v$ in $A$ joining
$a|I^{n-1}\times 1$ to $b|I^{n-1}\times 0$ and $b|I^{n-1}\times 1$ to $a|I^{n-1}\times 0$ such that the path
$a\ast u\ast b\ast v$ is homotopic to $a$ relatively to $ \partial I^n$.\\
2. The product $[a]\ast [b]$ of two homotopy classes is $[c]$,
where $c=a\ast v\ast b$, $v$ being a homotopy in $A$ joining $a|I^{n-1}\times 1$ and $b|I^{n-1}\times 0$.
\end{Definition}

\begin{Exercise}
The definition of $\pi_n(X,A)$ is sound.
\end{Exercise}

\begin{Exercise}
$\pi_n(X,A)$ is an Abelian group.
\end{Exercise}

\begin{Exercise}
If $B\subset A$, then the natural function $\pi_n(X,B)\to \pi_n(X,A)$ is an isomorphism. In particular, $\pi_n(X,x_0)\to \pi_n(X,A)$ is an isomorphism
for any $x_0\in A$.
\end{Exercise}

\begin{Corollary}
If $A\cap B\ne\emptyset$, then $\pi_n(X,A)$ is isomorphic
to $\pi_n(X,B)$.
\end{Corollary}

\end{document}